\documentclass[12pt,reqno]{amsart}

\usepackage{amsmath,amsfonts,amsbsy,amsgen,amscd,mathrsfs,amssymb,amsthm}
\usepackage{enumerate,mathtools}

\usepackage[usenames,dvipsnames]{xcolor}
\usepackage[colorlinks=true,citecolor=blue,linkcolor=blue]{hyperref}

\usepackage{tikz}
\usetikzlibrary{shadings}

\newtheorem{thm}{Theorem}[section]

\newtheorem{rem}[thm]{Remark}

\theoremstyle{definition}

\numberwithin{equation}{section} 
\numberwithin{figure}{section}
\numberwithin{table}{section}

\newcommand{\bE}{\mathbf{E}}
\newcommand{\eps}{\varepsilon}
\newcommand{\bR}{\mathbf{R}}

\newcommand{\bP}{\mathbf{P}}

\begin{document}

\title{Pisier type inequalities for  $K$-convex  spaces}



\author{Alexander Volberg}
\address{(A.V.) Department of Mathematics, MSU, 
East Lansing, MI 48823, USA}
\email{volberg@math.msu.edu}

\begin{abstract}
We generalize several theorems  of Hyt\"onen-Naor \cite{HN} using the approach from \cite{IVHV}. In particular, we give yet another necessary and sufficient condition (see \eqref{dF}) to be a $K$-convex space, where the sufficiency was proved by Naor--Schechtman \cite{NS}. This condition is in terms of the boundedness of the second order Riesz transforms 
$\{\Delta^{-1} D_i\}_{i=1}^n$ in $L^p(\Omega_n, X)$.
\end{abstract}



\thanks{ The research of the  author is supported by NSF DMS-1900286, DMS-2154402 and by Hausdorff Center for Mathematics }

\maketitle

\thispagestyle{empty}

\section{A variant of Pisier's inequality in spaces of finite co-type}
\label{sec:co-type}

In this note we consider functions on Hamming cube $\Omega_n:=\{-1, 1\}^n$. It is provided with natural measure giving $2^{-n}$ weight to each point, the integration will be denoted by $\bE$. There are differentiations: $\partial_i$ denotes the usual partial derivative with respect to $i$-th variable $\eps_i$, but $D_i$ is often more convenient:
$$
D_i=\eps_i \partial_i\,.
$$
Laplacian is $\Delta := D_1+\dots+D_n$. Notice that $D_i^2= D_i$.

\bigskip

The next theorem is basically proved in \cite{IVHV}, but  for the convenience of reading 
we prove it here as it plays an important part in what follows.

\begin{thm}
\label{co-type}
For any functions $f_i: \{-1,1\}^n \to X$, $i=1, \dots, n$, $p\in [1, \infty)$, we have
\begin{equation}
\label{Deltafi}
\bE\|\sum_{i=1}^n D_i f_i\|^p \le  C(q, p) \bE\|\sum_{i=1}^n \delta_i \Delta f_i\|^p
\end{equation}
if and only if $(X,\|\cdot\|)$ be a Banach space of {\it finite co-type} $q$.
\end{thm}

Here $\delta_i$ are standard Rademacher random variables.
If one chooses $f_i= \Delta^{-1} D_i (f-\bE f)$, then one restores  Proposition 4.2 of \cite{IVHV}. In other words, this becomes Pisier inequality \cite{P} with constant independent of $n$.

\begin{rem}
In \cite{HN} it was proved  for UMD Banach spaces (and a bit more general case).
\end{rem}

\begin{proof}
The proof of \eqref{Deltafi} for all finite co-type Banach spaces  $X$ follows almost immediately from the formula
\begin{equation}
\label{main}
e^{-t\Delta} D_j f (\eps)= \frac{e^{-t}}{\sqrt{1-e^{-2t}}} \bE_\xi \Big[\frac{\xi_j(t) -e^{-t}}{\sqrt{1-e^{-2t}}} f(\eps \xi(t))\Big]\,,
\end{equation}
where $\xi_i(t)$ assume values $\pm 1$ with probability $\frac12(1\pm e^{-t})$, and are mutually independent and independent from $\eps_j, j=1, \dots, n$. The formula can be checked by direct calculation, see also Lemma 2.1 of \cite{IVHV}.

Adding and applying $\Delta= \Delta_\eps$, we get (we denote $\delta_j(t):= \frac{\xi_j(t) -e^{-t}}{\sqrt{1-e^{-2t}}}$)

$$
\Delta e^{-t\Delta} \sum_{j=1}^n D_j f_j (\eps)= \frac{e^{-t}}{\sqrt{1-e^{-2t}}} \bE_\xi \Big[\sum_{j=1}^n \delta_j(t) [\Delta f_j](\eps \xi(t))\Big]\,,
$$
After integrating in $t$, we get
$$
-\sum_{j=1}^n D_j f_j (\eps) = \int_0^\infty  \bE_\xi \Big[\sum_{j=1}^n \delta_j(t) [\Delta f_j](\eps \xi(t))\Big]\, \frac{e^{-t}}{\sqrt{1-e^{-2t}}} dt\,.
$$
Hence, we write
$$
(\bE_\eps \|\sum_{j=1}^n D_j f_j (\eps)\|^p)^{1/p} \le \int_0^\infty \Big(\bE_\xi\bE_\eps \Big\|\sum_{j=1}^n \delta_j(t) [\Delta f_j](\eps \xi(t))\Big\|^p\Big)^{1/p} \, \frac{e^{-t}}{\sqrt{1-e^{-2t}}} dt=
$$
$$
 \int_0^\infty \Big(\bE_\xi\bE_\eps \Big\|\sum_{j=1}^n \delta_j(t) [\Delta f_j](\eps )\Big\|^p\Big)^{1/p} \, \frac{e^{-t}}{\sqrt{1-e^{-2t}}} dt,
$$
where we used that for every fixed $t, \xi(t)$ the distribution of $\eps\to [\Delta f_j](\eps\xi (t))$ is the same as that of $\eps\to [\Delta f_j](\eps )$. We continue by  introducing symmetrization by means of $\{\xi_i'(t)\}$ independent from all $\{\xi_j\}_{j=1}^n$ and having the same distribution as $\xi_i$, $i=1, \dots, n$,
$$
 \int_0^\infty \Big(\bE_\eps \bE_{\xi, \xi'} \Big\|\sum_{j=1}^n  \frac{\xi_j(t) -\xi'_j(t)}{\sqrt{1-e^{-2t}}} [\Delta f_j](\eps )\Big\|^p\Big)^{1/p} \, \frac{e^{-t}}{\sqrt{1-e^{-2t}}} dt\le
$$
$$
C(q, p)  \int_0^\infty \Big(\bE_\eps \bE_\delta \Big\|\sum_{j=1}^n \delta_j [\Delta f_j](\eps )\Big\|^p\Big)^{1/p} \, \frac{e^{-t}}{(1-e^{-2t})^{1-\frac{1}{\max (q, p)} }}dt\,,
$$
where we used Theorem 4.1 from \cite{IVHV} and the fact that the co-type of $X$ is $q<\infty$. The last expression is bounded by 
$$
C(q, p) \max (q, p)\Big( \bE_\eps \bE_\delta \Big\|\sum_{j=1}^n \delta_j [\Delta f_j](\eps )\Big\|^p\Big)^{1/p},
$$
 and we are done.

Notice that, as it follows from \cite{IVHV}, the condition of having finite co-type is not only sufficient for inequality \eqref{Deltafi} to hold, but it is also necessary.

\end{proof}

\section{Pisier inequality's constant}
\label{log}

For any Banach space $X$ with no restriction Pisier's inequality claims
\begin{equation}
\label{Pisier}
(\bE\|f-\bE f\|^p)^{1/p} \le  C(n) (\bE\|\sum_{i=1}^n \delta_i D_i f\|^p)^{1/p}\,,
\end{equation}
where
$$
C(n) \le C \log n\,.
$$
In \cite{HN} it was shown that
\begin{equation}
\label{logC}
C(n) \le  \log n + C\,.
\end{equation}

\begin{rem}
Looking at Section 6 in Talagrand's paper \cite{T} one can notice, that  one gets the estimate from below of Pisier's constant if $X=L^\infty(\Omega_n)$:
$$
C(n) \ge (\frac12 - \delta) \log n - C_\delta\,.
$$
\end{rem}

\section{Yet another generalization of Pisier's inequality}
\label{F}

Let $F$ be a function of $\{-1, 1\}^n\times \{-1, 1\}^n$ with values in the Banach space $X$, and 
$F_j(\eps) = \bE_\delta \delta_j F(\eps, \delta)$.  For the special case $F(\eps, \delta) =\sum_{j=1}^n f_j(\eps) \delta_j$, inequality
\begin{equation}
\label{F1}
(\bE_\eps\|\sum_{j=1}^n \Delta^{-1} D_j F_j\|^p)^{1/p} \le C(p, n) (\bE_{\delta, \eps} \|F\|^p)^{1/p},\, 1< p <  \infty,
\end{equation}
is exactly \eqref{Deltafi}. For such  very special $F=\sum_{j=1}^n f_j(\eps) \delta_j$ we know three things: 
\begin{enumerate}
\item  for general Banach space $X$, $C(p, n) \lesssim \log n$, 
\item this is sharp growth, 
\item  $C(p, n)\le C(p, q)<\infty$ iff $X$ is of finite co-type, and if $X$ is not of finite co-type,  constant  can grow logarithmically in $n$, \cite{T}.
\end{enumerate}

\bigskip

However, it is interesting to ask for general function $F(\eps, \delta)$ not just for functions of the type $F =\sum_{j=1}^n f_j(\eps) \delta_j$ what happens with \eqref{F1}, namely, 

A) for what Banach spaces constant does not depend on $n$? 

B) What is the worst growth of constant with $n$ for general Banach space $X$?

C) What is the worst growth of constant with $n$ for special classes of Banach spaces, e.g. for $X$ of finite co-type?

\bigskip

\begin{rem}
In \cite{HN} the example of co-type $2$ space is considered, namely, $X= L^1(\{-1, 1\}^n)$, for which constant grows at least as $\sqrt{n}$.  Below we show that this is the worst behavior for an arbitrary Banach space  {\it of finite co-type}. Thus, conceptually, \eqref{F1} turns out to be very different from  \eqref{Deltafi} or from the original Pisier inequality.  
\end{rem}

\bigskip

Before formulating theorem let us consider the dual inequality to \eqref{F1}:
\begin{equation}
\label{dF}
\bE_{\delta, \eps}\|\sum_{j=1}^n \delta_j \Delta^{-1} D_j g(\eps)\|^p \le C(p, n) \bE_\eps\|g\|^p, \, 1< p< \infty\,.
\end{equation}
In cases when $C(p, n)<\infty$ independent of $n$ and for $1<p<\infty$, this is one of the  {\it  typical Riesz transforms} inequalities.

\begin{thm}
\label{F1thm}
Let $X$ be of finite co-type $q$. Then inequality \eqref{F1} \textup(and thus \eqref{dF} for the dual space\textup) holds with constant $C(p, q)\sqrt{n}$.
The growth of constant cannot be improved in this class of $X$.
\end{thm}

\begin{proof}
We again use the same formula \eqref{main}, now in the following form:

$$
\Delta e^{-t\Delta} \sum_{j=1}^n \Delta^{-1}D_j F_j (\eps)= \frac{e^{-t}}{\sqrt{1-e^{-2t}}} \bE_\xi \Big[\sum_{j=1}^n \delta_j(t) [ F_j](\eps \xi(t))\Big]\,,
$$
Hence
$$
\bE_{\delta, \eps}\|\sum_{j=1}^n \Delta^{-1} D_j F_j\|^p \lesssim\int_0^\infty\frac{e^{-t}}{\sqrt{1-e^{-2t}}}\bE_\delta\bE_\xi\bE_\eps\|\sum_{j=1}^n \delta_j\delta_j(t) F(\eps\xi, \delta)\|^p dt=
$$
$$
\int_0^\infty\frac{e^{-t}}{\sqrt{1-e^{-2t}}}\bE_\delta\bE_\xi\bE_\eps\|\sum_{j=1}^n \delta_j\delta_j(t) F(\eps, \delta)\|^p dt,
$$
where we used that for every fixed $\xi$ the distribution of $\eps\to  F(\eps\xi, \delta)$ is the same as the distribution of $\eps\to  F(\eps, \delta)$.
Using now finite co-type as before, we continue to write (below $\{\delta_j'\}$ are independent and independent of $\{\delta_j\}$ Rademacher random variables):
\begin{equation}
\label{co-typeF}
\lesssim \int_0^\infty \frac{e^{-t}}{(1-e^{-2t})^{1- \min(1/p, 1/q)}} \bE_\delta\bE_\eps\bE_{\delta'} \|\sum_{j=1}^n \delta_j\delta_j' F(\eps, \delta)\|^p dt \lesssim
\end{equation}
$$
C(p, q) \bE_\delta\Big[\bE_\eps \|F(\eps, \delta)\|^p  \bE_{\delta'}|\sum_{j=1}^n  \delta_j\delta_j' |^p \Big] \le C'(p, q) n^{p/2}\bE_\eps \bE_\delta \|F(\eps, \delta)\|^p\,.
$$

Theorem is proved.
\end{proof}

\subsection{$K$-convex Banach spaces}
\label{K}

Naor and Schechtman \cite{NS} proved that if \eqref{dF} holds with constant independent of $n$, then $X$ is $K$-convex. However, it seems that the converse statement was  open till now.  

Let us quote \cite{EI1}: ``Nevertheless, \eqref{dF} with $C(p, n)=C (\log n+1)$ is the best known bound to date for general $K$-convex spaces\dots .   Under additional assumptions (e.g. when $X$ is a $UMD^+$ space or when $X$ is a $K$-convex Banach lattice), inequality \eqref{dF} is known to hold true with a constant $C(p, X)$ independent of the dimension $n$ for functions of arbitrary degree $d$, see \cite{HN}''.

Proposition 34 of \cite{EI1} has the bound for $K$-convex spaces, and this bound is logarithmic in $n$ (logarithmic in $d$ for functions $f$ such that $\deg f\le d$).  As \cite{EI1} mentions, under extra assumption that $X$ is a $UMD^+$ space or $K$-convex Banach lattice inequality \eqref{dF} was proved with constant independent of $n$, see \cite{HN}. 

We prove \eqref{dF} with constant independent of $n$ for all $K$-convex $X$, thus making $K$-convexity to be equivalent to \eqref{dF}.

\bigskip

Recall that $K$-convexity is equivalent to $B$-convexity, which is equivalent to being  of type $>1$, see Theorem  2.1  and Remark 2.4 of \cite{PAnnals}.

\begin{thm}
\label{F1type}
Let $X$ be of non-trivial type (which is the same as $K$-convex). Let $1< p<\infty$. Then inequality \eqref{dF}  holds with constant $C( p)<\infty$ independently of $n$. 
\end{thm}

\begin{proof}
We will be proving the dual inequality \eqref{F1} for $X^*$. As $X$ is of non-trivial type, it is $K$-convex by Pisier's theorem  (see \cite{PAnnals} or Theorem 7.4.28 of \cite {HVNVW2}) (and is of finite co-type by K\"onig--Tzafriri theorem 7.1.14 in \cite{{HVNVW2}}, this we will not use). Then $X^*$ is of finite co-type $q$ and it is also $K$-convex (see \cite{G} for self-duality of the class of $B$-convex Banach spaces).

Choose $1<s\le p$.  We will use that $K$-convexity means that the Rademacher projection is bounded on functions in $L^s(X^*)$.
As $X^*$ is of  finite co-type $q$, Theorem \ref{co-type}  implies  the following:
\begin{equation}
\label{F2}
\bE_\eps\|\sum_{i=1}^n \Delta^{-1}D_i F_i\|_{X^*}^p \le  C(q, p) \bE_{\delta, \eps}\|\sum_{i=1}^n \delta_i F_i\|_{X^*}^p\,.
\end{equation}
Now we use Kahane--Khintchine inequality to write for each fixed $\eps_0\in \Omega_n$:
\begin{equation}
\label{KK}
 \bE_\delta\|\sum_{i=1}^n \delta_i F_i(\eps_0)\|_{X^*}^p \le  C(s,q, p)\Big( \bE_\delta\|\sum_{i=1}^n \delta_i F_i(\eps_0)\|_{X^*}^s\Big)^{p/s}\,.
\end{equation}
But the expression $\sum_{i=1}^n \delta_i F_i(\eps_0)$ is the Rademacher projection of function $\delta\to F(\eps_0, \delta)$. 
So $K$-convexity of $X^*$ implies

$$
\Big(\bE_\delta\|\sum_{i=1}^n \delta_i F_i(\eps_0)\|_{X^*}^s \Big)^{p/s}\le C'(K) \Big(\bE_\delta\|F(\eps_0, \delta)\|_{X^*}^s \Big)^{p/s}\le 
$$
$$
C'(K) \bE_\delta\|F(\eps_0, \delta)\|_{X^*}^p \,.
$$
Now we combine that inequality with  \eqref{KK} for a fixed $\eps=\eps_0$. We are left to integrate in $\eps_0$ and to use \eqref{F2}.

\end{proof} 

\begin{rem}
In \cite{HN} inequality \eqref{F1} was proved for $X$ such that $X^*\in UMD$ \textup(in fact a potentially bigger class $UMD^+$ was involved\textup). As non-trivial-type class is self dual, and also strictly wider than $UMD$, so the latter theorem generalizes Theorem 1.4 of \cite{HN}.
\end{rem}
\begin{rem}
It is interesting to notice that the proof in \cite{HN} is based on a formula that means that operators $\Delta^{-1} D_j$ are 
``averages of martingale transforms". As these operators are ``the second order Riesz transforms" on Hamming cube \textup(in fact, $\Delta^{-1} D_j= \Delta^{-1} D_j^2= (\Delta^{-1/2}D_i)^2$\textup), it is natural to compare this averaging of martingale transforms to Riesz transforms with the same idea recently widely used in harmonic analysis, see, e. g. \cite{DV}, \cite{PTV}, \cite{NTV1}, \cite{NTV2}. Paper  \cite{DV} is devoted to representing second order Riesz transforms in euclidean space as averaging of martingale transforms, in \cite{PTV} the similar result is proved for the first order Riesz transforms.
\end{rem}

\begin{thm}
\label{F1log}
Let $X$ be an arbitrary Banach space, and $1\le p<\infty$. Then inequality \eqref{F1}  holds with constant $C(p, n)\lesssim n\log n$.
\end{thm}

\begin{proof}
We write
$$
(\bE_{\delta, \eps}\|P_\tau\sum_{j=1}^n \Delta^{-1} D_j F_j\|^p)^{1/p} \lesssim
$$
$$
\int_\tau^\infty\frac{e^{-t}}{\sqrt{1-e^{-2t}}}(\bE_\delta\bE_\xi\bE_\eps\|\sum_{j=1}^n \delta_j\delta_j(t) F(\eps\xi, \delta)\|^p)^{1/p} dt=
$$
$$
\int_\tau^\infty\frac{e^{-t}}{1-e^{-2t}}(\bE_\delta\bE_\xi\bE_\eps\|\sum_{j=1}^n \delta_j(\xi_j(t)-\xi_j'(t)) F(\eps, \delta)\|^p )^{1/p}dt,
$$
Now we use Kahane contraction principle:
$$
(\bE_{\delta, \eps}\|P_\tau\sum_{j=1}^n \Delta^{-1} D_j F_j\|^p)^{1/p} \lesssim\log\frac{1+e^{-\tau}}{1-e^{-\tau}} \Big[\bE_\delta \bE_\eps \Big(\|F(\eps, \delta)\|^p \cdot (|\sum_{j=1}^n \delta_j|^p)\Big)\Big]^{1/p} \le 
$$
$$
n\log\frac{1+e^{-\tau}}{1-e^{-\tau}} (\bE_{\eps, \delta} \|F(\eps, \delta)\|^p)^{1/p} \,.
$$
Now using that $\|f\|_p \le e^{\tau  n} \|P_\tau f\|_p$, we get
$$
(\bE_{\delta, \eps}\|\sum_{j=1}^n \Delta^{-1} D_j F_j\|^p)^{1/p} \le n\Big(\min_{0<r<1} r^{-pn} \log \frac{1+r}{1-r} \Big)^{1/p} (\bE_{\eps, \delta} \|F(\eps, \delta)\|^p)^{1/p}
$$
$$
\lesssim n\log n (\bE_{\eps, \delta} \|F(\eps, \delta)\|^p)^{1/p}\,.
$$
\end{proof}

\begin{rem}
This theorem sounds a bit silly. It should be $\lesssim \sqrt{n}$ for all Banach spaces.  It can be that we missed something simple. On the other hand, it may be a worthwhile exercise to ``marry" the example giving $\sqrt{n}$ in \cite{HN} and Talagrand's example from \cite{T}, to possibly have a Banach space with behavior of constant in \eqref{F1}, which is worse than $\sqrt{n}$. I did not try so far.
\end{rem}

\end{document}